\documentclass[11pt,leqno]{amsart}
\usepackage{amsmath, amsthm, amssymb,xspace}
\usepackage[breaklinks=true]{hyperref}
\usepackage{amsmath,amscd}
\theoremstyle{plain}
\newtheorem{theorem}{Theorem}[section]
\newtheorem{lemma}{Lemma}[section]

\theoremstyle{definition}

\newcommand{\mf}[1]{\displaystyle{\mathfrak{#1}}}

\newcommand{\comment}[1]{}

\begin{document}

\title[commutants of monomials]{Commutants of Toeplitz operators with  monomial symbols }

\author{Akaki Tikaradze}
\address{Department of Mathematics, Universiy of Toledo, Toledo, Ohio, USA}
\email{\tt tikar06@gmail.com}
\begin{abstract}

In this note we describe the commutant of the multiplication operator by  a monomial 
in the Toeplitz algebra of a complete strongly pseudoconvex Reinhardt domain.

\end{abstract}
\maketitle

Throughout given an $n$-tuple of nonnegative integers $\alpha=(i_1, \cdots, i_n)$ and 
$z=(z_1,\cdots,z_n)\in\mathbb{C}^n,$ we will
denote monomial $z_1^{i_1}\cdots z_{n}^{i_n}$ by $z^{\alpha}.$ Also, $\mathbb{Z}_+$ will denote the set of
all nonnegative integers.

Recall that a bounded domain $\Omega\subset \mathbb{C}^n$ is said to be a complete Reinhardt domain
if $(z_1,\cdots, z_n)\in \Omega$ implies $(a_1z_1, \cdots, a_nz_n)\in \Omega$
for any complex numbers $a_i\in \bar{D},  1\leq i\leq n,$ where $D=\lbrace z\in \mathbb{C}: |z|<1\rbrace$ denotes the unit disk.

As usual, the Bergman space of square integrable holomorphic functions on a bounded domain $\Omega\subset \mathbb{C}^n$
is denoted by $A^2(\Omega).$
 Recall that given a bounded measurable function $f\in L^{\infty}(\Omega),$
one defines the corresponding Toeplitz operator $T_f:A^2(\Omega)\to A^2(\Omega)$ with symbol $f$
 as $T_f(\phi)=P(f\phi),  \phi\in A^2(\Omega)$, where 
$P:L^2(\Omega)\to A^2(\Omega)$ is the orthogonal projection.
If the symbol $f$ is holomorphic, then $T_f$ is the multiplication operator on $A^2(\Omega)$ with symbol $f.$
Let $\mf{T}(\Omega)$ denote the $C^*$-algebra generated by $\lbrace T_g : g \in  L^{\infty}(\Omega)\rbrace.$
We will refer to this algebra as the Toeplitz  $C^{*}$-algebra of $\Omega.$

Given a Toplitz operator $T_f$, it is of great interest to study the commutant of $T_f$ in the Toeplitz $C^{*}$-algebra of $\Omega.$ 

To this end, in the case of the unit disk $\Omega=D$ in $\mathbb{C}$,  Z. Cuckovic [\cite{Cu}, Theorem 1.4] showed that
  the commutant of $T_{z^k}, k\in \mathbb{N} $ in the Toeplitz $C^{*}$-algebra of $D$
  consists of the Toeplitz operators with
bounded holomorphic symbols. Subsequently T. Le [\cite{Le}, Theorem 1.1] has generalized Cuckovic's result
  to the case of  of the unit ball $\Omega=B_n\subset \mathbb{C}^n$ and a monomial $f=z_{1}^{m_1}\cdots z_{n}^{m_n}$ 
 such that $m_i>0$ for all $1\leq i\leq n.$
  
In this note we extend Le's  result to the case of strongly pseudoconvex complete Reinhardt domains. (Theorem \ref{Le}.)
Moreover, our proof is simpler and computation free.

As in \cite{Cu}, \cite{Le}, the following result plays the crucial role in the proof of
Theorem \ref{Le}.

\begin{theorem}\label{Trieu}

Let $\Omega\subset \mathbb{C}^n$ be a bounded complete Reinhardt domain, and let
$f={z_1}^{m_1}{z_2}^{m_2}\cdots {z_n}^{m_n} $  be a monomial such that $m_i>0$ for all $1\leq i\leq n.$
 If $S:A^2(\Omega)\to A^2(\Omega)$ is a compact
operator that commutes with  $T_f,$ then $S=0.$

\end{theorem}

We will need the following trivial

\begin{lemma}\label{calculus}

Let $\phi\in L^{\infty}(\Omega).$ Then there exists $\epsilon>0$
such that $\int_{\Omega}|\phi|^{m+1}d\mu\geq \epsilon \int_{\Omega}|\phi|^md\mu,$ for all $m\in \mathbb{N}.$

\end{lemma}

\begin{proof}

Without loss of generality we may assume that $\mu({\phi}^{-1}(0))=0.$
Put $\Omega_{\epsilon}=|\phi|^{-1}((0, \epsilon)).$ Choose $\epsilon>0$ so that
$\mu(\Omega_{\epsilon})\leq\frac{1}{2}\mu(\Omega).$ Then for all $m\in\mathbb{N}$
$$\int_{\Omega}|\phi|^{m+1}d\mu\geq \int_{\Omega\setminus \Omega_{\epsilon}}|\phi|^{m+1}d\mu\geq \epsilon \int_{\Omega\setminus \Omega_{\epsilon}}|\phi|^md\mu.$$
But 
$$\int_{\Omega_{\epsilon}}|\phi|^md\mu\leq \epsilon^m\mu(\Omega_{\epsilon})\leq \int_{\Omega\setminus \Omega_{\epsilon}}|\phi|^md\mu.$$ 
Therefore
$$\int_{\Omega}|\phi|^{m+1}d\mu\geq \frac{1}{2}\epsilon\int_{\Omega}|\phi|^md\mu.$$

\end{proof}

\begin{proof}(Theorem \ref{Trieu}.)
Let us write $f=z^{\tau}, \tau=(m_1, \cdots, m_n).$ Since $\Omega$ is a complete Reinhardt domain,
it is well-known that monomials
$\lbrace z^{\gamma}, \gamma\in \mathbb{Z}_{+}^{n}\rbrace$  form an orthogonal
basis of $A^2(\Omega).$ Assume that there exits a nonzero compact operator $S:A^2(\Omega)\to A^2(\Omega)$ that
commutes with $T_{z^{\tau}}.$ Thus $S(gz^{m\tau})=S(g)z^{m\tau},$ for all $g\in A^2(\Omega), m\in \mathbb{N}.$
 Let $\alpha, \beta\in \mathbb{Z}_{+}^n$ be
such that $\langle S(z^{\alpha}), z^{\beta} \rangle_{A^2(\Omega)}\neq 0.$ Write 
$S(z_{\alpha})=\sum_{\gamma} c_{\gamma}z^{\gamma}, c_{\gamma}\in \mathbb{C}.$
 Thus $c_{\beta}\neq 0.$
It follows that for any $m\in \mathbb{N}$
$$\|S(z^{{\alpha}+m\tau})\|_{A^2(\Omega)}=\|z^{m\tau}S(z^{\alpha})\|_{A^2(\Omega)}\geq |c_{\beta}|\|z^{\beta+m\tau}\|_{A^2(\Omega)}.$$
 Let $\beta'\in \mathbb{N}^n, k\in \mathbb{N}$ be such that  $z^{\beta}z^{\beta'} =z^{k\tau}.$
Such  $\beta', k$ exist because $m_i>0,$ for all $1\leq i\leq m.$  
Let  $\epsilon'>0$ be such that $\epsilon'||gz^{\beta'}||_{A^2(\Omega)}\leq ||g||_{A^2(\Omega)}$ for all $g\in A^2(\Omega).$ 
  Then for $\epsilon=|c_{\beta}|\epsilon'>0,$ we have
$$\|S(z^{\alpha+m\tau})\|_{A^2(\Omega)}\geq \epsilon \|z^{(m+k)\tau}\|_{A^2(\Omega)}, m\geq 0.$$
By  Lemma \ref{calculus}, there exists $\delta'>0$ such that for all $m\geq 0$ 
$$\|z^{(m+k)\tau}\|_{A^2(\Omega)}\geq \delta'\|z^{m\tau}\|_{A^2(\Omega)}.$$
Put $\delta=\epsilon\delta'.$ Combining the above inequalities we get that for all $m\geq 0$
$$\|S(z^{\alpha+m\tau})\|_{A^2(\Omega)}\geq \delta \|z^{m\tau}\|_{A^2(\Omega)}.$$
On the other hand since
$$||z^{\alpha+m\tau}||_{A^2(\Omega)}\leq ||z^{\alpha}||_{L^{\infty}(\Omega)}||z^{m\tau}||_{A^2(\Omega)},$$ 
we get that  
$$\frac{\|S(z^{\alpha+m\tau})\|_{A^2(\Omega)}}{\|z^{\alpha+m\tau}\|_{A^2(\Omega)} }\geq \frac{\delta}{\|z^{\alpha}\|_{L^{\infty}(\Omega)}}, m\in \mathbb{N}.$$

However, the sequence 
$\lbrace\frac{z^{\alpha+m\tau}}{\|z^{\alpha+m\tau}\|_{A^2(\Omega)}}\rbrace, m\in \mathbb{N}$ converges to 0 weakly.
Thus compactness of $S$  implies that 
$$\lim_{m\to\infty}\frac{\|S(z^{\alpha+m\tau})\|_{A^2(\Omega)}}{\|z^{\alpha+m\tau}\|_{A^2(\Omega)}}=0,$$ a contradiction.

\end{proof}

In the proof of Theorem \ref{Le} we will use the Hankel operators. Recall that 
given a function $\phi\in L^{\infty}(\Omega)$, the Hankel operator $H_{\phi}:A^2(\Omega)\to L^2(\Omega)$
with symbol $\phi$ is defined by $H_{\phi}(g)=\phi g-P(\phi g), g\in A^2(\Omega).$

The proof of the following uses Theorem \ref{Trieu} and is essentially the same as in [\cite{Cu}, page 282]. 

\begin{theorem}\label{Le}
Let $\Omega\subset \mathbb{C}^n$ be a bounded smooth strongly pseudoconvex complete Reinhardt domain,
and let $f=z_1^{m_1}\cdots z_n^{m_n}, m_i>0$ be a monomial. 
If $S$ is an element of the Toeplitz $C^{*}$-algebra of $\Omega$  which commutes
with $T_{f},$ then $S$ is a multiplication operator by
a  bounded holomorphic function on $\Omega.$
\end{theorem}

\begin{proof}

Recall that for any $g\in L^{\infty}(\Omega)$ and a holomorphic 
$\psi\in A^{\infty}(\Omega)$ we have $[T_g, T_{\psi}]={H_{\bar{\psi}}}^*H_g.$ On the other hand
$H_{\bar{z_i}}$ is a compact operator for all $1\leq i\leq n,$  as easily follows  from \cite{Pe}.  Thus, $[T_g, T_{z_i}]$ is a 
compact operator
for all $1\leq i\leq n, g\in L^{\infty}(\Omega)$.
This implies that  the commutator $[s, T_{z_i}]$ is compact for any $s\in \mf{T}(\Omega), 1\leq i\leq n.$
 If $S\in \mf{T}(\Omega)$ commutes with $T_f,$ then
so do compact operators $[S, T_{z_i}], 1\leq i\leq n.$ Therefore, by Theorem \ref{Le} we have 
$[S, T_{z_i}]=0, 1\leq i\leq n.$ This implies that $S=T_g$ for some bounded holomorphic $g$
by [\cite{SSU}, proof of Theorem 1.4].

\end{proof}

\end{document}